\documentclass[a4paper,11pt]{article}
\usepackage[sc]{mathpazo}
\usepackage[T1]{fontenc}
\usepackage{amsthm,amsmath,amssymb,amsfonts,hyperref,a4wide}

\usepackage{color}
\definecolor{blue}{RGB}{0,0,255}
\newcommand{\new}[1]{\textcolor{black}{#1}}
\usepackage{soul}
\usepackage[usenames,dvipsnames]{xcolor} 
\usepackage{tikz}
\makeatletter
\tikzset{my loop/.style =  {to path={
  \pgfextra{}
  [looseness=12,min distance=6mm]
  \tikz@to@curve@path},font=\sffamily\small
  }}  
\makeatletter 
\usetikzlibrary{matrix,arrows,backgrounds,positioning,fit,shapes}

\newtheorem{theorem}{Theorem}
\newtheorem{lemma}[theorem]{Lemma}

\newtheorem*{cora}{Corollary 6a}
\newtheorem*{corb}{Corollary 6b}

\theoremstyle{definition}
\newtheorem{ex}{Example}
\newtheorem*{df}{Definition}
\theoremstyle{remark}
\newtheorem{rem}{Remark}

\newtheorem*{ack}{Acknowledgements}

\newcommand*{\R}{\mathbb{R}}

\newcommand*{\N}{\mathbb{N}}
\newcommand*{\C}{\mathbb{C}}

\newcommand*{\G}{\mathcal{G}}

\newcommand*{\ind}{\text{ind}}

\newcommand*{\eps}{\varepsilon}

\title{Zero-free regions of partition functions with applications to algorithms and graph limits}
\author{Guus Regts\footnote{University of Amsterdam. Email: \texttt{guusregts@gmail.com.}}}
\begin{document}
\maketitle
\abstract{
Based on a technique of Barvinok \cite{B14a,B14b,B15} and Barvinok and Sober\'on \cite{BS14a,BS14b} we identify a class of edge-coloring models whose partition functions do not evaluate to zero on bounded degree graphs. 
Subsequently we give a quasi-polynomial time approximation scheme for computing these partition functions. 
As another application we show that the normalised partition functions of these models are continuous with respect \new{to} the Benjamini-Schramm topology on bounded degree graphs. 
We moreover give quasi-polynomial time approximation schemes for evaluating a large class of graph polynomials, including the Tutte polynomial, on bounded degree graphs.

\begin{footnotesize}
Keywords: Benjamini-Schramm convergence, approximation algorithms, partition functions, sparse graph limits.
\end{footnotesize}
}
\section{Introduction}
In a series of papers Barvinok \cite{B14a,B14b,B15} and Barvinok and Sober\'on \cite{BS14a,BS14b} developed a technique that yields quasi-polynomial time approximation schemes for several hard counting problems such as computing the permanent of a matrix and computing the number of (edge-colored) graph homomorphisms into a fixed graph.
Their technique is essentially based on two things. 
First the counting problem is \new{cast} as the evaluation of some polynomial and they identify a region where this polynomial is nonzero. 
Secondly, they prove that in this zero-free region the logarithm of the polynomial is well-approximated by a low order Taylor approximation, which they prove to be computable in quasi-polynomial time.

In this paper we will apply this method to partition functions of edge-coloring models and more generally to Holant counting problems and contractions of tensor networks. 
That is, fixing a degree bound $\Delta$, we will identify edge-coloring models for which the partition function does not vanish on graphs with maximum degree at most $\Delta$ and consequently show how to approximately compute the partition function in quasi-polynomial time. 
We moreover show that the normalised logarithm of the partition function is continuous with respect to the Benjamini-Schramm topology for these models, and we give give quasi-polynomial time algorithms for approximately evaluating a large class of graph polynomials, including the Tutte polynomial and the matching polynomial, on bounded degree graphs. 

Below we will introduce and describe these concepts in a bit more detail and state our results.

\subsection{Partition functions of edge-coloring models}
Edge-coloring models originate in statistical physics and their partition functions have been introduced to the graph theory community by de la Harpe and Jones \cite{HJ} (where they are called vertex models).
We call any map $h:\N^k\to \C$ a \emph{$k$-color edge-coloring model}. 
(In this paper we will fix $k\geq2$ and just speak about edge-coloring models).
For a graph $G=(V,E)$, the \emph{partition function} of $h$ is defined by 
\begin{equation}\label{eq:def pf}
p(G)(h):=\sum_{\phi:E\to[k]}\prod_{v\in V}h(\phi(\delta(v))),
\end{equation}
\new{where $\delta(v)$ denotes the multiset of edges incident with the vertex $v$ and} where $\phi(\delta(v))$ denotes the multiset of colors that the vertex $v$ `sees', which we identify with its incidence vector in $\N^k$ so that we can apply $h$ to it. 
Let us give some examples.
\begin{ex}\label{eq:pf}
\begin{itemize}
\item
Let $h:\N^2\to\C$ be defined by $h(\new{\alpha_1,\alpha_2})=1$ if $\alpha_1\leq 1$ and zero otherwise for $\new{(\alpha_1,\alpha_2)}\in \N^2$.
Then $p(G)(h)$ is equal to the number of matchings of the graph $G$.
\new{This can be easily seen as follows: the edges receiving color $1$ correspond to a matching if and only if the coloring gives a contribution to the sum} \eqref{eq:def pf}.
\item \new{Fix $d\in \N$ and} let $h:\N^{2}\to \C$ be defined by $h(\new{\alpha_1,\alpha_2})=1$ if $\alpha_1=d$ and zero otherwise for $\new{(\alpha_1,\alpha_2})\in \N^{2}$.
Then $p(G)(h)$ is equal to the number of $d$-regular subgraphs of $G$ on the same vertex set.
\end{itemize}
See \cite{S7} and below for more examples.
\end{ex}

Interchanging the role of edges and vertices in \eqref{eq:def pf} one obtains partition functions of vertex-coloring models.
Let $a\in \C^n$ such that $a_i\neq 0$ for all $i$ and let $B\in \C^{n\times n}$ be a symmetric matrix. We call the pair $(a,B)$ a \emph{vertex-coloring model} (In \cite{HJ} this is called a \emph{spin model} provided $a$ is the all ones vector). For a graph $G=(V,E)$ the partition function of $(a,B)$ is defined by
\begin{equation}\label{eq:def pf vcm}
p(G)(a,B):=\sum_{\phi:V\to[n]}\prod_{v\in V}a_{\phi(v)}\prod_{uv \in E}B_{\phi(u),\phi(v)}.
\end{equation}
If $a$ is the all-ones vector and $B$ is the adjacency matrix of some graph $H$, then $p(G)(a,B)$ is equal to the number of graph homomorphisms from $G$ to $H$.

Partition functions of edge-coloring models include partition function of vertex-coloring models, as has been shown by Szegedy \cite{S7,S10}, \new{cf. Lemma \ref{lem:edge vertex} in Section \ref{sec:def}.} 
But they form a much bigger class of graph invariants, cf. \cite{DGLRS}. For example, the number of matchings is not equal to the partition function of any vertex-coloring model.
Partition functions of edge-coloring models have been characterised in \cite{S7} and \cite{DGLRS} based on invariant theory of the orthogonal group and the Nu\new{l}lstellensatz. 

In many cases it is NP-hard to determine whether or not the partition function of an edge- or vertex-coloring model vanishes on general graphs. 
However, when restricted to bounded degree graphs things may become easier.
In particular, Barvinok and Sober\'on \cite{BS14a} proved that when $a$ is the all ones vector and $B$ is sufficiently close to the all ones matrix, then 
\begin{equation}
\text{$p(G)(a,B)\neq0$ for all graphs of maximum degree at most $\Delta$.}\label{eq:BS vcm}
\end{equation}
Based on the technique of Barvinok \cite{B14a,B14b,B15} and Barvinok and Sober\'on \cite{BS14a,BS14b} we will prove in Section \ref{sec:absence} a similar result for partition functions of edge-coloring models:
\begin{theorem}\label{thm:absence main}
\new{Let $\Delta\in \N$ then there exists a constant $\beta>0.35$ such that for any $k$-color edge-coloring model $h$ that satisfies $|h(\alpha)-1|\leq \beta/(\Delta+1)$ for all $\alpha\in \N^k$ with $|\alpha|\leq \Delta$ we have for any graph with maximum degree at most $\Delta$, $p(G)(h)\neq 0$.}
\end{theorem}
In fact, we will state and prove in Section \ref{sec:absence} a slightly more general version of Theorem \ref{thm:absence main} that also applies to tensor networks. The definition of tensor networks will be given in Section \ref{sec:prel}.

\begin{rem}\label{rem:vertex-coloring}
While every partition function of a vertex-coloring model is equal to the partition function of some edge-coloring model, the exact relation between Theorem \ref{thm:absence main} and \eqref{eq:BS vcm} is presently not clear. 
It is possible though to derive a qualitative version of \eqref{eq:BS vcm} from Theorem \ref{thm:absence main} (even valid for vertex vertex-coloring models with vertex weights close to $1$). We will however not do this here. 
Instead, in Section \ref{sec:absence} we will present an example of a vertex-coloring model to which the result of Barvinok and Sober\'on does not apply, while from our results it follows that its partition function does not evaluate to zero on bounded degree graphs, cf. Example \ref{ex:ind}.
\end{rem}

\subsection{Computing partition functions}
{\bf Exact computation} Computing partition functions is in many cases $\#$P-hard. 
There is a line of work by Dyer and Greenhill \cite{DG00}, Bulatov and Grohe \cite{BG05}, Cai, Chen and Lu \cite{CCL13}, where dichotomy theorems have been proved for the complexity of computing partition functions of vertex-coloring models. 
Roughly they proved that if the model has a very special structure, its partition function can be computed in polynomial time and it is $\#$P-hard to compute otherwise.
In a sequence of papers Cai Lu and Xia \cite{CLX09,CLX11}, Cai, Huang and Li \cite{CHL10}, and Cai, Guo and Williams \cite{CGW13} have proved similar results for partition functions of edge-coloring models with $2$ colors.\footnote{In fact, in these papers they consider \emph{Holant counting problems}, but they include partition functions of edge-coloring models. These Holant counting problems may be seen as tensor network contractions and we will define these in Section \ref{sec:tensor}.} \\
\quad \\
\noindent {\bf Approximate computation}
Since computing partition functions is in many case $\#$P-hard, people have started looking at approximation algorithms. 
Let us consider the case of counting weighted independent sets.
For a graph $G$ and $\lambda>0$, denote by $p_\lambda(G)$ the sum over all independent sets of $G$, where an independent set $I$ is counted with weight $\lambda^{|I|}$. 
Note that $p_{\new{\lambda}}(G)$ may be viewed as the partition function of $a=(\new{\lambda},1)$ and 
$B=\left (\begin{array}{cc}0&1\\1&1\end{array}\right)$. 
A surprising threshold phenomenon has been found for the complexity of \new{approximating} $p_\lambda(G)$: for $\lambda<\lambda_c(\Delta):=(\Delta-1)^{\Delta-1}/(\Delta-2)^{\Delta}$ there exists a fully-polynomial time approximation scheme (FPTAS) for graphs of maximum degree at most $\Delta$, discovered by Weitz \cite{W6}; for $\lambda>\lambda_c$ there does not exist a fully-polynomial time randomised approximation scheme (FPRAS) for $\Delta$-regular graphs (with $\Delta\geq 3$), unless NP=RP, as has been proved by Sly and Sun \cite{SS12}. 
The threshold $\lambda_c$ is the critical point for the (infinite) $\Delta$-regular tree to have a unique Gibbs measure. 
This phenomenon suggests an intimate connection between phase transitions in statistical physics and computational complexity.
\new{ The ideas of Weitz have been applied and extended to other types of counting problems and partition functions; see e.g. \cite{BGKNT7,GK12,LY13,SST14}. Other interesting approximation algorithms and hardness results have been found for partition functions of vertex-coloring models an planar graphs \cite{GJM15}. }
\\
\quad \\
\noindent{\bf Our contribution}
In Section \ref{sec:algorithms} of this paper we apply Theorem \ref{thm:absence main} to obtain a quasi-polynomial time approximation scheme (QPTAS) for the computation of partition functions of a large class of edge-coloring models on bounded degree graphs, based on a technique developed by Barvinok \cite{B14a}. 
This technique was used by Barvinok and Sober\'on \cite{BS14a} to obtain a QPTAS for the computation of partition function of vertex-coloring models on bounded degree graphs (with $a$ equal to the all ones vector).

The main idea is to rewrite the partition function as the evaluation of a univariate polynomial $q$ at $1$ and then use a low order Taylor approximation to $\ln q$. 
The approximation turns out to be good when the roots of $q$ have absolute value at least $M$ for some fixed constant $M>1$. 
Theorem \ref{thm:absence main} is used to guarantee the existence of this constant $M$.
From the perspective of the Lee-Yang approach to phase transition in statistical physics \cite{LY52}, one may view zero-free regions of the partition function as regions far away from a phase transition. 
So this approach also indicates a connection between statistical physics and computational complexity.

We need one definition to state our result.
\begin{df}
Let $\eps>0$ and $a\in \C$. We call $b\in \C$ a \emph{multiplicative $\eps$-approximation} to $a$ if 
$e^{-\eps}\leq \left |\frac{a}{b}\right|\leq e^{\eps}$ and if the angle between $a$ and $b$ is at most $\eps$.
We call $b\in \C$ an \emph{additive $\eps$-approximation} to $a$ if $|a-b|\leq \eps$.
\end{df}
We prove the following result in Section \ref{sec:algorithms}:
\begin{theorem}\label{thm:algorithms}
Let $\Delta\in \N$. Then there exists a constant $\delta>0$ (we may take $\delta=0.35/(\Delta+1)$) such that for any $\eps>0$, any edge-coloring model $h$ that satisfies $|h(\alpha)-1|\leq \delta$ for each $\alpha\in \N^k$ and any graph $G=(V,E)$ of maximum degree at most $\Delta$, we can compute a multiplicative $\eps$-approximation to $p(G)(h)$ in time $|V|^{O(\ln |V|-\ln \eps)}$ and an additive $\eps|V|$-approximation to $\ln|p(G)(h)|$ in time $|V|^{O(\ln1/\eps)}$.
\end{theorem}
Theorem \ref{thm:algorithms} also applies to contractions of tensor networks and Holant counting problems. 
See Section \ref{sec:algorithms} for the exact details.

\subsection{Evaluating graph polynomials}
It turns out that the technique of Barvinok can also be applied to graph polynomials of exponential type. 
This is a class of graph polynomials introduced by Csikv\'ari and Frenkel \cite{CF12}, which includes many important polynomials such as the Tutte polynomial, the matching polynomial and many more.
We will define these polynomials in Section \ref{sec:exp type}, where we obtain a QPTAS for computing certain evaluations of these polynomials on bounded degree graphs, cf. Theorem \ref{thm:exp type approx}.
Let us for now specialise these results to the Tutte polynomial.

For a graph $G = (V, E)$ the random cluster model formulation of the \emph{Tutte polynomial} $Z(G)(\new{q}, v)$, is defined as follows:
\begin{equation}
Z(G)(\new{q}, v) := \sum_{A\subseteq E}\new{q}^{k(A)}v^{|A|}, 	\label{eq:tutte}
\end{equation}
where $k(A)$ denotes the number of connected components of the graph $(V, A)$. 
We will consider the Tutte polynomial for fixed $v$ as a polynomial in $\new{q}$.
\begin{theorem}\label{thm:tutte}
Let $\Delta\in \N$ and let $v_0\in \C$. Then there exists a constant $\new{C=C(\Delta,v_0)}$ such that for any $\eps>0$, any graph $G=(V,E)$ of maximum degree at most $\Delta$ and any $q\in \C$ with $|q|>C$, we can compute a multiplicative $\eps$-approximation to $Z(G)(q,\new{v_0})$ in time $|V|^{O(\ln|V|-\ln \eps)}$ and an additive $\eps|V|$-approximation to $\ln |Z(G)(q,\new{v_0})|$ in time $|V|^{O(\ln1/\eps)}$.
\end{theorem}
In \cite{GJ12} Goldberg and Jerrum showed that approximating $Z(G)(\new{q},v)$ for general graphs and for $\new{q}>2$ and $v>0$ is as hard as counting independent sets in bipartite graphs (\#BIS). 
The present results may be an indication that this is not the case for bounded degree graphs.
See \cite{LY13} for related algorithmic results for the Potts model and \cite{GSVY14} for hardness results for the Potts model on bounded degree graphs.

\subsection{Sparse graph limits}
In 2001, Benjamini and Schramm \cite{BS11} defined a convergence notion for finite \new{bounded degree planar} graphs, which \new{can also be extended to general bounded degree graphs} and is often referred to as \emph{Benjamini-Schramm convergence}, \emph{local convergence} or \emph{weak convergence}.
This convergence notion has many equivalent definitions (see the book by Lov\'asz \cite[Proposition 5.6]{L12}) and we choose one which is most suitable for our purposes. 
Fix a degree bound $\Delta$ and let $(G_n)$ be a sequence of simple graphs, each of maximum degree at most $\Delta$.
We call the sequence $(G_n)$ \emph{locally convergent} if $|V(G_n)|$ goes to infinity, as $n\to\infty$ and if for each connected simple graph $H$ the quantity
\begin{equation}
\frac{\text{ind}(H,G_n)}{|V(G_n)|}	\label{eq:def convergence}
\end{equation}
is a convergent sequence of real numbers, where in \eqref{eq:def convergence} $\text{ind}(H,G_n)$ denotes the number of induced subgraphs of $G_n$ that are isomorphic to $H$.
As an example let $C_n$ be the cycle of length $n$. 
Then $(C_n)$ is locally convergent. 
See \cite{L12} for possible representations of limit objects.

Let $\G_\Delta$ be the collection of simple graphs of maximum degree at most $\Delta$.
Call a graph parameter $f:\G_\Delta\to \C$ \emph{estimable} if for any locally convergent sequence of graph\new{s} $(G_n)\subset \G_{\Delta}$, $(f(G_n))$ is a convergent sequence of numbers.
In Section \ref{sec:sparse limits} we look at the estimability of the \emph{normalised partition function} $n(G)(h)$ of an edge-coloring model $h$, which is defined as  
\begin{equation}
n(G)(h):=\frac{\ln|p(G)(h)|}{|V(G)|},\label{eq:normalised pf}
\end{equation}
where $G$ is a graph.
This normalisation is motivated by statistical physics, where this quantity is known as the \emph{entropy per vertex}.
Note from the perspective of estimability it does not matter to take the absolute value of $p(G)(h)$ before applying the logarithm. For if we fix a some branch of the (complex) logarithm, then the imaginary part is bounded, and dividing it by $|V(G_n)|$ will make it disappear in the limit. 

Borgs, Chayes, Kahn and Lov\'asz \cite{BCKL12} showed that if one in \eqref{eq:normalised pf} replaces $p(G)(h)$ by $p(G)(a,B)$ for appropriate $(a,B)$ this parameter is estimable; see also \cite{LML15}.
In \cite{CF12} Csikv\'ari and Frenkel showed that if one replaces $p(G)(h)$ by $q(G)(z)$ in \eqref{eq:normalised pf}, where $q$ is a graph polynomial of bounded exponential type and $|z|$ is large enough, then this parameter is estimable. 
This extends work of Ab\'ert and Hubai \cite{AH12} who showed that this holds for the chromatic polynomial.

Utilising Theorem \ref{thm:absence main} and building on a result of Csikv\'ari and Frenkel \cite{CF12}, we will prove that the normalised partition function of an edge-coloring model that is sufficiently close to the all ones edge-coloring model is also estimable. 
\begin{theorem}\label{thm:sparse limit}
\new{Let $\Delta\in \N$. Then there exists a constant $\beta>0.35$ such that for any edge-coloring model $h$ that satisfies $|h(\alpha)-1|\leq \beta/(\Delta+1)$ for each $\alpha\in \N^k$, the normalised partition function $n(\cdot)(h):\G_\Delta\to \R$ is estimable.}
\end{theorem}

\subsection{Organisation of the paper}
In the next section we collect some preliminaries and set up some notation.
In particular, we introduce tensor networks.
In Section \ref{sec:absence} we will prove a slight generalisation of Theorem \ref{thm:absence main} valid for tensor networks.
In Section \ref{sec:algorithms} we will discus\new{s} the technique of Barvinok \cite{B14a} mentioned earlier and prove Theorems \ref{thm:algorithms} and \ref{thm:tutte}. 
In Section \ref{sec:sparse limits} we will give a proof of Theorem \ref{thm:sparse limit}.
Finally, in Section \ref{sec:remarks} we will conclude with a few remarks.
Section\new{s} \ref{sec:algorithms} and \ref{sec:sparse limits} may be read independently of each other.

\section{Preliminaries}\label{sec:prel}
\subsection{Notation and definitions}\label{sec:def}
We will setup some basic notation and give some definitions here that are used throughout the paper.

By $\G$ we denote the collection of all graphs, allowing multiple edges and loops.
Recall that by $\G_\Delta$ we denote the collection of simple graphs of maximum degree at most $\Delta$.
Let $S$ be a ring. An \emph{$S$-valued graph invariant} or \emph{$S$-valued graph parameter} is a map $f:\G\to S$ that maps isomorphic graphs to the same element of $S$. 
We often omit the reference to $S$.
When $S=\C[z]$ (the polynomial ring in the variable $z$) we call $f$ a \emph{graph polynomial}.
A \emph{graph invariant} $f$ is called \emph{multiplicative} is $f(G\cup H)=f(G)f(H)$ where $G\cup H$ denotes the disjoint union of two graphs $G,H$. Clearly, partition functions of edge-coloring models are multiplicative.

Whenever we talk about an edge-coloring model we will always assume it has $k$ colors, where $k$ is some fixed natural number. 
In particular, in algorithmic statements we will consider $k$ as a constant and moreover assume that it is at least $2$, for else computing the partition function is trivial.

\new{
Let $U$ be any $k\times n$ matrix and let $a\in \C^n$.
Let $u_1,\ldots,u_n$ be the columns of $U$.
Define the $k$-color edge-coloring model $h_{a,U}$ by 
\begin{equation}
h_{a,U}(\alpha_1,\ldots,\alpha_k):=\sum_{i=1}^n a_i\prod_{j=1}^k u_i(j)^{\alpha_j} \label{eq:edge vertex}
\end{equation}
for $\alpha\in \N^k$.
In case $a$ is the all ones vector we just write $h_U$ instead of $h_{a,U}$.
Szegedy \cite{S7,S10} showed the following\footnote{While Szegedy only proved this for real matrices $B$, it is easily seen that the proof remains valid for complex matrices, cf. \cite[Lemma 7.1]{R13}.}:
\begin{lemma}\label{lem:edge vertex}
Let $B$ be an $n\times n$ complex symmetric matrix and let $a\in (\C\setminus\{0\})^n$.
Let $U$ be any (complex) matrix such that $U^TU=B$.
Then for any graph $G$ we have $p(G)(a,B)=p(G)(h_{a,U})$.
\end{lemma}}

\subsection{Tensor networks}\label{sec:tensor}
As mentioned in the introduction, Theorem \ref{thm:absence main} not only applies to edge-coloring models but also to tensor networks. We will introduce tensor networks here.

First we need some further definitions.
For $\alpha\in \N^k$ set $|\alpha|:=\sum_{i=1}^k \alpha_i$.
By $\N^k_d$ we denote the set $\{\alpha\in \N^k:|\alpha|=d\}$.
We identify the set of functions $\{h:\N^k_d\to \C\}$ with $\C^{\N^k_d}$ \new{and denote it by $S_d$}. 
Let $G=(V,E)$ be graph. We define 
\[
S_G=\prod_{v\in V}\new{S_{\deg(v)}^v}.
\]
Let $h=(h^v)_{v\in V}\in S_G$. We call the pair $(G,h)$ a \emph{tensor network}.
The \emph{contraction} of the tensor network $(G,h)$ is defined as follows:
\begin{equation}
p(G)(h):=\sum_{\phi:E\to[k]}\prod_{v\in V}h^v(\phi(\delta(v))).\label{eq:tensor network}
\end{equation}
This is almost the same as the partition function of an edge-coloring model. 
The main difference is that in the tensor network we can select a function for each vertex, while for an edge-coloring model the function is only allowed to depend on the degree of the vertex.
It will be convenient to view an edge-coloring model $h$ as an infinite vector $(h^d)_{d\in \N}\in \prod_{d\in \N} \C^{\N^k_d}$. 

The reason why $(G,h)$ is called a tensor network comes from the fact that we may view $h^v$ as a symmetric tensor in $(\C^k)^{\otimes \deg(v)}$ for each $v\in V$. \new{The contraction of the tensor network then corresponds to contracting a tensor obtained from these $h^v$'s along the edges of $G$ with respect to a bilinear form.
We refer to Section 3.4 of the author's PhD thesis \cite{R13} for more details on this perspective.}

Contractions of tensor networks have been used to simulate quantum computing by Markov and Shi in \cite{MS08}, where they gave a polynomial time algorithm for contracting tensor networks on bounded degree graphs of bounded treewidth. 
In \cite{AL10} Arad and Landau use quantum algorithms to compute additive approximations to tensor network contractions. We refer to both \cite{MS08} and \cite{AL10} and the references in there for more details concerning the connection of tensor networks and quantum information theory.
Contractions of tensor networks can also be used to model Holant counting problems (see \cite{CLX09,CHL10,CLX11}).

\subsection{The orthogonal group}
There is a natural action of the complex orthogonal group, $O_k(\C)$ (the group of $k\times k$ matrices $g$ such that $g^Tg=I$ where $I$ is the identity matrix) on tensor networks (and hence on edge-coloring models) that leaves the tensor network contraction invariant.
\new{
To define this action in a convenient way, we need to realize that $\{h:\N^k\to \C\}$ is in one-to-one correspondence with the space of linear map from $\C[x_1,\ldots,x_k]$ to $\C$. 
Each such map is determined by its action on monomials and the monomials are in one-to-one correspondence with $\N^k$. 
This we may view any $h\in \C^{\N^k_{d}}$ as a linear map $h:\C[x_1,\ldots,x_k]\to \C$ that is zero on all monomials of total degree bigger than $d$.
Now $O_k(\C)$ acts on $\C[x_1,\ldots,x_k]$  as follows: an element $g\in O_k(\C)$ sends a polynomial $p$ to the unique polynomial $gp$ that on input $v\in \C^k$ takes the value $p(g^Tv)$.
Then $O_k(\C)$ also naturally acts on the space of linear maps $\C[x_1,\ldots,x_k]\to \C$ via $gf(p):=f(g^Tp)$ for $g\in O_k(\C)$  $p\in \C[x_1,\ldots,x_k]$, $v\in \C^k$ and any linear map $f:\C[x_1,\ldots,x_k]\to \C$.}

\new{
For a tensor network $(G,h)$ with $h=(h^v)_{v\in V}\in S_G$ and $g\in O_k(\C)$ we define $(G,gh)$ to be the tensor network with $gh:=(gh^v)_{v\in V}\in S_G$.
Then}
\begin{equation}
\text{for any $g\in O_k(\C)$, graph $G$ and $h\in S_G$: $p(G)(gh)=p(G)(h)$.} \label{eq:orthogonal invariance}
\end{equation}
\new{For edge-coloring models this is proved in \cite{DGLRS} and the same proof also works for contractions of tensor networks.}

\section{Absence of roots for tensor network contractions}\label{sec:absence}
In this section we will find zero-free regions for tensor network contractions. In particular, Theorem \ref{thm:absence main} will be proved here.
We begin with some notation, necessary to state our main theorem. 
Define for $d\in \N$, $\delta\in (0,1)$ and $\eta>0$, the set
\[
S_d(\delta,\eta):=\{h\in \C^{\N^k_d}\mid |h_\alpha-h_{\alpha'}|<\delta\ \text{ and } |h_\alpha|\geq \eta, \text{ for all }\alpha,\alpha' \in \N^k_d \},
\]
and for a graph $G=(V,E)$, the set
\[
S_G(\delta,\eta):=\prod_{v\in V} S^v_{\deg(v)}(\delta,\eta)\subset S_G.
\]
We can now state our main theorem.

\begin{theorem}\label{thm:absence of roots}
Let $G=(V,E)$ be a graph and let $\eta\in (0,\infty)$.
For any $\theta\in(0,2\pi/3)$ fix $\beta\leq \eta\theta\cos(\theta/2)$ and let $\delta:=\min\{\eta,\frac{\beta}{\Delta(G)+1}\}$.
Then for each $h\in S_G(\delta,\eta)$, $|p(G)(h)|\geq (\cos(\theta/2)\eta)^{|V|}k^{|E|}$. 
In particular, $p(G)(h)\neq 0$.
\end{theorem}
We will prove Theorem \ref{thm:absence of roots} in the next section, but we will first state some consequence.

Obviously, the goal is to choose $\theta$ with $\beta$ as large as possible.
For this we need to solve $\max_{\theta\in (0,2\pi/3)}\theta\cos(\theta/2)$. 
This maximum is attained by the solution $\theta^*$ to the equation $2/\theta=\tan(\theta/2)$, which is approximately equal to $\theta^*\approx 1.72067$. Let $x^*=\theta^*\cos(\theta^*/2)$. Then $x^*\approx 1.12219$.
Define \new{for $d\in \N$ 
\[\beta^*(d):=\frac{x^*}{1+\frac{x^*}{2d}}.\] 
Then $ 0.71885\approx\beta^*(1)<\beta^*(2)<\beta^*(3)<\cdots \leq x^*\approx 1.12219$.}
We have the following corollary, which implies Theorem \ref{thm:absence main}. 

\begin{cora}\label{cor:absence of roots}
Let $G=(V,E)$ be a graph. Fix nonzero $a_v\in \C$ for each $v\in V$. \new{Write $\beta^*:=\beta^*(\Delta(G)+1)$.}
Let $h\in S_G$ such that $|a_v-h^v(\alpha)|\leq \frac{|a_v|\beta^*}{2(\Delta(G)+1)}$ for all $v\in V$ and $\alpha\in \N_{\deg(v)}^k$. Then 
\begin{equation}|p(G)(h)|\geq \prod_{v\in V}|a_v|\cdot \big(\cos(\theta^*/2)(1-\frac{\beta^*}{2\Delta(G)+2})\big)^{|V|}k^{|E|}.\label{eq:bound pf}
\end{equation}
In particular, $p(G)(h)\neq 0$.
\end{cora}
\begin{proof}
Let us first consider the case where $a_v=1$ for all $v$. Let $\eta:=1-\frac{\beta^*}{2\Delta(G)+2}$.
A simple computation shows that $\new{\beta^*=\eta x^*}$.
Clearly, $\eta>\beta^*/2$. So if we let $\delta=\frac{\beta^*}{\Delta(G)+1}$, then each $h^v\in S(\delta,\eta)$ and hence \eqref{eq:bound pf}, with $a_v=1$ for each $v\in V$, follows from Theorem \ref{thm:absence of roots}.

In the general case let $g^v(\alpha)=h^v(\alpha)/a_v$ for each $\alpha $ and $v\in V$. 
Then for each $v$ and $\alpha$, $|1-g^v(\alpha)|\leq \frac{\beta^*}{2(\Delta(G)+1)}$.
So \eqref{eq:bound pf}, with $a_v=1$ for each $v\in V$, holds for $p(G)(g)$.
Since $p(G)(h)=\prod_{v\in V}{a_v}\cdot p(G)(g)$, \eqref{eq:bound pf} now follows.
\end{proof}
We obtain another corollary if we take the action of the orthogonal group into account.
\new{Let $x\in \C^k$ and consider $x$ as a $k\times 1$ matrix. 
Let $h_x$ denote the edge-coloring model as defined in \eqref{eq:edge vertex} (so $n=1$, $u_1=x$ and $a_1=1$ in this case).}
\begin{corb}\label{cor:invarance}
Let $G=(V,E)$ be a graph. Then for each $x\in \C^k$ such that $(x,x)=\sum_{i=1}^kx_i^2\neq 0$ and nonzero $a_v\in \C$ for each $v\in V$, there exists a constant $\delta>0$ (only depending on $x$ \new{and $\Delta(G)$}) such that if $h\in S_G$ satisfies 
\[
|h^v(\alpha)-\new{a_vh_x}(\alpha)|\leq \frac{|a_v||(x,x)/k|^{\deg(v)/2}\delta}{2(\Delta(G)+1)}
\] 
for all $v\in V$ and $\alpha\in \N^k_{\deg(v)}$, then $p(G)(h)\neq 0$.
\end{corb}
\begin{proof}
Since $(x,x)\neq 0$ there exist $g\in O_k(\C)$ such that $gx=\lambda\mathbb{1}$, a scalar multiple of the all-ones vector, with $\lambda=((x,x)/k)^{1/2}$.
\new{Then for any $\alpha\in \N^k_d$ we have that  $gh_x(\alpha)=\lambda^{d}$.
By continuity, if $h^v$ is close to $a_vh_x$ it follows that $gh^v(\alpha)$ is close to $a_v\lambda^{d}$ for each $\alpha\in \N^k_d$.}
The result now follows from Corollary \ref{thm:absence of roots}a and the fact that tensor network contraction is invariant under the action of the orthogonal group.
\end{proof}
\new{Corollary \ref{thm:absence of roots}b is roughly saying that if the tensors $h^v$ are close enough to the rank one tensor $h_x$, then the tensor network contraction is not equal to zero.}

We will now give an example illustrating the power and weakness of our results.
\begin{ex}\label{ex:ind}
Consider for $\lambda,\mu\in \C$ the matrix $B=\left( \begin{array}{cc} 0&\lambda\mu\\ \lambda\mu&\mu^2\end{array}\right)$ and let $a=(1,1)$.
Since the matrix $B$ contains a zero, the result of Barvinok and Sober\'on \cite{BS14a} does not apply to it.
However, if we let $U=1/\sqrt{2}\left(\begin{array}{cc}(1+i)\lambda&\mu\\ (1-i)\lambda &\mu\end{array}\right)$, then $U^TU=B$. 
\new{Consider the edge-coloring model $h_U$ as defined in \eqref{eq:edge vertex}. 
Then, by Lemma \ref{lem:edge vertex}}, we have $p(G)(a,B)=p(G)(h_U)$ for each graph $G$.
By Corollary \ref{thm:absence of roots}a, for fixed nonzero $\mu$, fixed $\Delta>0$ and small enough $\lambda$ we have that $p(G)(h_U)\neq 0$ for all graphs with maximum degree at most $\Delta$.

For $\mu=1$ we can view $p(h_U)(G)$ as a polynomial in $\lambda$, which for $\Delta$-regular graphs coincides with the \emph{independence polynomial} of $G$ evaluated at $\lambda^\Delta$.
Scott and Sokal \cite{SS05} proved that, if $|\lambda^\Delta|\leq \frac{(\Delta-1)^{\Delta-1}}{\Delta^{\Delta}}$, then $p(h_U)(G)\neq 0$ on any $\Delta$-regular graph $G$. They moreover \new{showed} that this bound is tight.
As $\Delta \to \infty$ this bound behaves roughly like $\frac{1}{e\Delta}$, while the result we obtain is of the order $O(\frac{1}{\Delta^\Delta})$, showing that our results are not optimal.
\end{ex}

\subsection{A proof of Theorem \ref{thm:absence of roots}}
Our proof of Theorem \ref{thm:absence of roots} is based on a method of Barvinok \cite{B14a,B14b,B15} and Barvinok and Sober\'on \cite{BS14a,BS14b}. 
Following Barvinok \cite{B15} we can say that although the method of proof is similar to the methods considered in \cite{B14a,B14b,BS14a,BS14b,B15}  it does require some effort and new ideas. 

We will first state and prove several lemmas.
The first lemma about angles between vectors is due to Boris Bukh see \cite{B14b}. 
We refer to \cite{B14b} for a proof.
\begin{lemma}\label{lem:angle bound 1}
Let $\vartheta\in [0,2\pi/3)$. Let $x_1,\ldots,x_n$ be nonzero vectors in $\R^2$ such that for each pair $i,j$ the angle between $x_i$ and $x_j$ does not exceed $\vartheta$.
Then $\|\sum_{i=1}^nx_i\|\geq\cos(\vartheta/2)\sum_{i=1}^n \|x_i\|$.
\end{lemma}

We will consider elements of $\C$ as vectors in $\R^2$ so that we can speak about the angle between two complex numbers. This is how Lemma \ref{lem:angle bound 1} will be applied.

It will be convenient to introduce some notation.
Let $F,F'$ be finite sets with $F\subseteq F'$. Let $\phi:F\to [k]$ and let $\psi:F'\to [k]$. 
We say that $\psi$ \emph{extends} $\phi$ if $\psi|_F=\phi$.
When $G=(V,E)$ is a graph, $u\in V$, $F\subseteq E$ and $\alpha\in \N^k_{\deg(u)}$,  we say that $\alpha$ is \emph{$u$-compatible} with $\phi$ if there exists a map $\phi':F\cup \delta(u)\to [k]$ that extends $\phi$ such that $\phi'(\delta(u))=\alpha$. We denote this by $\phi \sim_u \alpha$. 

It will also be convenient to consider \new{the} coordinate ring of $\C^{\N^k_d}$, which is the polynomial ring $R_d:=\C[x_{\alpha}\mid \alpha\in \N^k_d]$. 
For a graph $G=(V,E)$ we define 
\begin{equation}
R_G:=\prod_{v\in V} R^{\new{v}}_{\deg(v)}.\nonumber
\end{equation}
For $F\subseteq E$ and $\phi:F\to [k]$ define
\begin{equation}
p^F_\phi(G):=\sum_{\substack{\psi:E\to [k] \\\psi|_F=\phi}}\prod_{v\in V} x^v_{\psi(\delta(v))}\in R_G.
\label{eq:pf restrict}
\end{equation}
So $p^F_{\phi}(G)$ is a polynomial and we can evaluate this at $h\in S_G$.
In particular, if $F=\emptyset$, then $p^F_{\phi}(G)(h)$ just coincides with the ordinary tensor network contraction $p(G)(h)$.

In the two lemmas below we assume that we have fixed some graph $G=(V,E)$ and $\eta,\delta>0$.
Moreover, for $u\in V$ we let $\delta(u)\subseteq E$ be the set of edges incident with $u$ and $N(u)\subset V$ the set of neighbours of $u$.

\begin{lemma}\label{lem:angle bound}
Let $\tau>0$, let $F\subseteq E$, let $\phi:F\to [k]$ and let $u\in V$.
If for all $h\in S_G(\delta,\eta)$ and for each $\psi:F\cup \delta(u)\to [k]$ extending $\phi$ we have $p^{F\cup \delta(u)}_{\psi}(G)(h)\neq 0$
and moreover for each $v\in N(u)\cup\{u\}$:
\begin{equation}
|p^{F\cup \delta(u)}_{\psi}(G)(h)|\geq \tau\sum_{\substack{\alpha\in \N^k_{\deg(v)}\\ \psi\sim_v \alpha}}|h^v_{\alpha}|\left | \frac{\partial}{\partial{x^v_{\alpha}}}p^{F\cup\delta(u)}_{\psi}(G)(h)    \right|,\label{eq:assumption 2} 
\end{equation}
then the angle of $p^{F\cup\delta(u)}_{\psi'}(G)(h)$ and $p^{F\cup\delta(u)}_{\psi}(G)(h)$ for any $\psi,\psi':F\cup \delta(u)\to [k]$ extending $\phi$ is at most 
\begin{equation}
\frac{\delta(\Delta(G)+1)}{\tau \eta}.\label{eq:angle}
\end{equation}
\end{lemma}

\begin{proof}
Let us write $F'=F\cup \new{\delta(u)}$ and fix $\psi:F'\to [k]$.
Since $p^{F'}_{\psi}(G)\neq 0$ on $S_G(\delta,\eta)$ we can consider a branch of $\ln (p^{F'}_{\psi}(G))$ on $S_G(\delta,\eta)$.
For any $v\in N(u)\cup \{u\}$ and $\alpha\in \N^k_{\deg(v)}$ we have
\begin{equation}
\frac{\partial}{\partial{x^v_{\alpha}}}\ln (p^{F'}_{\psi}(G))=\frac{\frac{\partial}{\partial{x^v_{\alpha}}}p^{F'}_{\psi}(G)}{p^{F'}_{\psi}(G)}.\label{eq:partial one}
\end{equation}
Note that \eqref{eq:partial one} equals zero if $\alpha$ is not $v$-compatible with $\psi$.
Then for each $h\in S_G(\delta,\eta)$ we obtain by \eqref{eq:assumption 2} that
\begin{equation}
\sum_{\substack{\alpha\in \N^{k}_{\deg(v)}\\ \psi\sim_v \alpha}}\left |\frac{\partial}{\partial{x^v_{\alpha}}}\ln (p^{F'}_{\psi}(G)(h))\right |\leq \frac{1}{\tau\min\{|h_\alpha|\mid h\in S_G(\delta,\eta),\alpha\in \N^k_{\deg(v)}\}}\leq \frac{1}{\tau\eta}. \label{eq:log bound}
\end{equation}

Next, let $\psi': F'\to [k]$ be another map that extends $\phi$ and define an element $g=\prod_{v\in V}g^v \in S_G(\delta,\eta)$ as follows.
If $v\notin N(u)\cup\{u\}$, then $g^v_{\alpha}=h^v_{\alpha}$ for all $\alpha\in \N^k_{\deg(v)}$.
If $v\in N(u)$, let $E(uv)$ denote the set of edges connecting $u$ and $v$ ($G$ might have multiple edges). 
Consider $\beta=\psi(E(uv))$ and $\beta'=\psi'(E(uv))$ as elements of $\N^k$.
Then for each $\alpha\in \N^k_{\deg(v)}$, set $\alpha':=\alpha-\beta+\beta'$. 
Then define $g^v_\alpha:=h^v_{\alpha'}$ if $\alpha'\geq 0$ and $g^v_{\alpha}=h^v_{\alpha}$ otherwise.
Finally, set $g^u_\alpha=h^u_{\alpha}$ if $\alpha \neq \psi(\delta(u))$ and $g^u_\alpha=h^u_\beta$ otherwise, where $\beta=\psi'(\delta(u))$.
Note that $p^{F'}_{\psi'}(G)(h)=p^{F'}_{\psi}(G)(g)$.
This implies by the \new{fundamental theorem of calculus} that $|\Im \ln(p^{F'}_{\psi}(G)(h))-\Im \ln(p^{F'}_{\psi'}(G)(h))|
$ is bounded by 
\[
\max_{z\in S_G(\delta,\eta)}
\sum_{v\in N(u)\cup\{u\}}\sum_{\substack{\alpha\in \N^{k}_{\deg(v)}\\\psi\sim_v\alpha}}
\left |\frac{\partial}{\partial{x^v_{\alpha}}}\ln (p^{F'}_{\psi}(G)(z))\right |\cdot|g^v_{\alpha}-h^v_{\alpha}|\leq \frac{\delta(\Delta(G)+1)}{\tau \eta},
\]
where $\Im f$ denotes the imaginary part of a function $f$ and where the inequality is due to \eqref{eq:log bound} and the fact that $|h^v_{\alpha}-h^v_{\alpha'}|\leq \delta$ for any $\alpha,\alpha'\in \N^k_{\deg(v)}$.
This proves the lemma.
\end{proof}

\begin{lemma}\label{lem:pf bound}
Let $\theta\in [0,2\pi/3)$. Let $u\in V$ and let $F\subseteq E$ with $\delta(u)\subseteq F$.
Let $\phi:F\to[k]$.
Suppose that for each $v\in N(u)\cup\{u\}$, all $h\in S_G(\delta,\eta)$ and all maps $\psi,\psi':F\cup\delta(v)\to [k]$ extending $\phi$, we have $p^{F\cup\delta(v)}_{\psi}(G)(h)\neq 0$ and that the angle between $p^{F\cup \delta(v)}_{\psi}(G)(h)$ and $p^{F\cup \delta(v)}_{\psi'}(G)(h)$ is at most $\theta$.
Then for each $v\in N(u)\cup \{u\}$ and all $h\in S_G(\delta,\eta)$ we have
\begin{equation}
|p^{F}_\phi(G)(h)|\geq \cos(\theta/2) 
\sum_{\substack{\alpha\in \N^k_{\deg(v)}\\\phi\sim_v \alpha}} |h^v_{\alpha}|\left |\frac{\partial}{\partial{x^v_{\alpha}}}p^{F}_{\phi}(G)(h) \right |. \label{eq:pf bound} 
\end{equation}
\end{lemma}

\begin{proof}
First note that \eqref{eq:pf bound} is trivially true for $v=u$ (since we could replace $\cos(\theta/2)$ by $1$).
Now let $v\in N(u)$ and let $\alpha\in \N^k_{\deg(v)}$. 
If $\alpha$ is not $v$-compatible with $\phi$, then $\frac{\partial}{\partial x^v_{\alpha}}p^F_\phi(G)=0$; if $\alpha$ is $v$-compatible with $ \phi$, then
\begin{equation}
\frac{\partial}{\partial x^v_{\alpha}}p^F_\phi(G)=\frac{1}{x^v_\alpha}\sum_{\substack{\psi:F\cup \delta(v)\to [k]\\\psi|_F=\phi,\psi(\delta(v))=\alpha}}p^{F\cup \delta(v)}_\psi(G).\label{eq:pd pf}
\end{equation}
Since 
\[
p^F_\phi(G)=\sum_{\substack{\psi:F\cup \delta(v)\to [k]\\\psi|_F=\phi}}p^{F\cup \delta(v)}_\psi(G),
\]
we obtain that for any $h\in S_G(\delta,\eta)$,
\begin{align}
|p^F_\phi(G)(h)|&\geq \cos(\theta/2)\sum_{\substack{\psi:F\cup \delta(v)\to [k]\\\psi|_F=\phi}}\big|p^{F\cup \delta(v)}_\psi(G)(h)\big|		\label{eq:pd pf2}
\\
&\geq \cos(\theta/2)\sum_{\substack{\alpha\in \N^k_{\deg(v)}\\ \phi\sim_v\alpha}}
 \bigg|\sum_{\substack{\psi:F\cup \delta(v)\to [k]\\\psi|_F=\phi,\psi(\delta(v))=\alpha}}p^{F\cup \delta(\new{u})}_\psi(G)(h)\bigg|	\nonumber
 \\
&=
\cos(\theta/2)\sum_{\substack{\alpha\in \N^k_{\deg(v)}\\ \phi\sim_v\alpha}}
|h^u_\alpha|\left|\frac{\partial}{\partial x^v_{\alpha}}p^F_\phi(G)(h)\right|,\nonumber
\end{align}
where the first inequality in \eqref{eq:pd pf2} is due to Lemma \ref{lem:angle bound 1},
since by assumption for each $\psi,\psi'$ and any $h\in S_G(\delta,\eta)$ we have that the angles 
between $p^{F\cup \delta(v)}_\psi(G)(h)$ and $p^{F\cup \delta(v)}_{\psi'}(G)(h)$ differ by no more than $\theta$, the second inequality is obvious and the equality in \eqref{eq:pd pf2} is due to \eqref{eq:pd pf}. This proves the lemma.
\end{proof}

We can now give a proof of Theorem \ref{thm:absence of roots}.
\begin{proof}[Proof of Theorem \ref{thm:absence of roots}]
Fix a graph $G=(V,E)$ and $h\in S_G(\delta,\eta)$.
If we plug in $\tau=\cos(\theta/2)$ in \eqref{eq:angle}, then by assumption we have that \eqref{eq:angle} is at most $\theta$.
We will first show that the following three statements hold:
\begin{equation}\label{eq:statements}
\begin{array}{cl}
\text{(i)}&\text{For each $F\subseteq E$, for each $\phi:F\to [k]$ and for all $h\in S_G(\delta,\eta)$ we have} 
\\&p^F_\phi(G)(h)\neq 0.
\\
\text{(ii)}&\text{For each $F\subseteq E$ and each $u\in V$ let $\phi:F\to [k]$ and write $F'=F\cup \delta(u)$.}
\\&
\text{Then for each $\psi,\psi':F'\to [k]$ extending $\phi$ and each $h\in S_G(\delta,\eta)$} 
\\&\text{the angle between $p^{F\new{'}}_\psi(G)(h)$ and $p^{F\new{'}}_{\psi'}(G)(h)$ is at most $\theta$.}
\\
\text{(iii)}& \text{For each $F\subseteq E$ and $u\in V$ such that $\delta(u)\subseteq F$ let $\phi:F\to [k]$.}
\\&\text{Then for all $h\in S_G(\delta,\eta)$ and each $v\in N(u)\cup \{u\}$ we have }
\\&\quad \quad \quad |p^{F}_\phi(G)(h)|\geq \cos(\theta/2) \sum_{\substack{\alpha\in \N^k_{\deg(v)}\\\phi\sim_v \alpha}} |h^v_{\alpha}|\left |\frac{\partial}{\partial{x^v_{\alpha}}}p^{F}_{\phi}(G)(h) \right |.
\end{array}
\end{equation}

Note that \eqref{eq:statements} (i) for $F=\emptyset$ already implies that $p(G)(h)\neq 0$ for all $h\in S_G(\delta,\eta)$.

We will now prove that \eqref{eq:statements} (i), (ii) and (iii) hold for all sets $F\subseteq E$ by induction on $|E\setminus F|$.
If $F=E$, then (i) holds since $p^E_\phi(G)(h)$ is just the product of $|V|$ nonzero numbers.
Clearly, (ii) also holds.
Moreover, (iii) also holds since for each $v$, 
\[
\sum_{\substack{\alpha\in \N^k_{\deg(v)}\\\phi\sim_v \alpha}} h^v_{\alpha}\frac{\partial}{\partial{x^v_{\alpha}}}p^{E}_{\phi}(G)(h)=p^E_{\phi}(G)(h)
\]
\new{and the sum in this equation consists of only one term.}

Let now $F$ be a subset of $E$ of size strictly smaller than $E$.
Let $u\in V$.
If $\delta(u)\subseteq F$, then (ii) is clearly true. So we may assume that $\delta(u)\new{\nsubseteq} F$. 
Let $\phi:F\to [k]$ and let $\psi,\psi':F\cup\delta(u)\to [k]$ be maps that extend $\phi$. 
Then by induction from (i) and (iii) (since $|F\cup \delta(u)|>|F|$) we know that the conditions of Lemma \ref{lem:angle bound} are satisfied with $\tau=\cos(\theta/2)$. This implies that (ii) holds for the set $F$.

To prove that (i) holds for $F$, fix $v$ such that $\delta(v)\new{\nsubseteq} F$.
Then 
\[p^F_\phi(G)(h)=\sum_{\substack{\psi:F\cup\delta(v)\to [k]\\ \psi_{|_F}=\phi}}p^{F\cup\delta(v)}_{\psi}(G)(h).
\]
As by induction the $p^{F\cup\delta(v)}_{\psi}(G)(h)$ are nonzero and by (ii) we know that their angles differ by at most $\theta$, Lemma \ref{lem:angle bound 1} implies that $p^F_\phi(G)(h)$ is nonzero for any $h\in S_G(\delta,\eta)$.

Finally, take $u$ such that $\delta(u)\subseteq F$. If no such $u$ exists there is nothing to prove for (iii).
Fix $v\in N(u)\cup \{u\}$. 
Since the sum in \eqref{eq:pf bound} has only one contributing factor if $\delta(v)\subseteq F$, we may assume that $\delta(v)\new{\nsubseteq} F$.
By induction, from (ii) we know that the angle between $p^{F\cup \delta(v)}_{\psi}(G)(h)$ and $p^{F\cup \delta(v)}_{\psi'}(G)(h)$ for any $\psi,\psi':F\cup \delta(v)\to [k]$ extending $\phi$ is at most $\theta$, as $|F\cup \delta(v)|>|F|$. So Lemma \ref{lem:pf bound} implies that (iii) holds for $F$.
So by induction we conclude that \eqref{eq:statements} holds for all sets $F$.

To finish the proof write $V=\{v_1,\ldots,v_n\}$ and let $F_0=\emptyset$. 
For $i>0$ let $F_i:=F_{i-1}\cup \delta(v_i)$.
We claim that for each $i=0,\ldots,n$, $\phi:F_i\to [k]$ and $h\in S_G(\delta,\eta)$ we have
\begin{equation}
|p^{F_i}_\phi(G)(h)|\geq k^{|E\setminus F_i|}\eta^n (\cos(\theta/2))^{n-i}.	\label{eq:inductive bound}
\end{equation}
Indeed, for $i=n$ this is clearly true.
Next, pick any $i<n$ and $\phi:F_i\to[k]$. Then
\begin{align*}
|p^{F_i}_\phi(G)(h)|&=\bigg|\sum_{\substack{\psi:F_{i+1}\to [k]\\\psi_{|_{F_i}}=\phi}}p^{F_{i+1}}_\psi(G)(h)\bigg|
\geq \cos(\theta/2)\sum_{\substack{\psi:F_{i+1}\to [k]\\\psi_{|_{F_i}}=\phi}}|p^{F_{i+1}}_\psi(G)(h)|
\\
&\geq \cos(\theta/2)k^{|F_{i+1}\setminus F_i|}k^{|E\setminus F_{i+1}|}\cos(\theta/2)^{n-(i+1)}\eta^n,
\end{align*}
where the first inequality follows from \eqref{eq:statements} (ii), and the second by induction.
This shows \eqref{eq:inductive bound} and finishes the proof. 
\end{proof}

\section{Approximation algorithms}\label{sec:algorithms}
In this section we will consider algorithms for computing tensor network contractions and evaluations of graph polynomials.
Both algorithms are based on a method of Barvinok \cite{B14a} (which has also been used by Barvinok \cite{B14b,B15} and Barvinok and Sober\'on \cite{BS14a,BS14b}), which we describe first.
\subsection{Approximating evaluations of polynomials}
Let $p\in \C[z]$ be a polynomial of degree $d$ and suppose that $p(z)\neq 0$ for all $z$ in a closed disk of radius $M$.
In this section we will describe the method of Barvinok \cite{B14a} to find a multiplicative approximation to $p(t)$ for $t\in \C$, which is only based on knowing the values of $\frac{d^m}{dt^m}p|_{t=0}$.

Define the univariate function $f$ by
\begin{equation}
f(z):=\ln p(z)
\end{equation}
for $z\in \C$, where we fix a (continuous) branch of logarithm by fixing the principal branch of the logarithm for $p(0)$.

To approximate $p$ at $t\in \C$, we will find an additive approximation to $f$ at $t$ using the the Taylor expansion around $z=0$. This can then be transformed to give a multiplicative approximation to $p$.

Let
\begin{equation}
T_n(f)(t):=f(0)+\sum_{m=1}^n \left.\frac{t^m}{m!}\frac{d^m}{dz^m}f(z)\right|_{z=0}.
\label{eq:taylor}
\end{equation}
Note that \eqref{eq:taylor} only depends on the values of $\frac{d^m}{dt^m}p|_{t=0}$ for $m=0,\ldots,n$. (Where we agree that the $0^{\text{th}}$ derivative is just the function itself.)
To see this note that $f'(z)= p'(z)/p(z)$, that is, $p'(z)=p(z)f'(z)$. So for $m\geq 1$ we have
\begin{equation}
\left.\frac{d^m}{dz^m}p(z)\right|_{t=0}=\sum_{j=0}^{m-1}\binom{m-1}{j}\left(\left.\frac{d^j}{dz^j}p(z)\right|_{z=0}\right)\left(\left.\frac{d^{m-j}}{dz^{m-j}}f(z)\right|_{z=0}\right). \label{eq:derivative}
\end{equation}
This implies that if we can compute the values of $\frac{d^j}{dz^j}p(z)|_{z=0}$ for $j=0,\ldots,n$, then \eqref{eq:derivative} provides a nondegenerate (as $p(0)\neq 0$) triangular system of equations to compute $\frac{d^j}{dz^j}f(z)|_{z=0}$ for $j=1,\ldots,n$, which can be done in time $n^{O(1)}$.
To summarise:
\begin{lemma}\label{lem:p implies f}
If the values $\frac{d^m}{dz^m}\new{p}|_{z=0}$ are given for $m=0,\ldots,n$, then $\frac{d^m}{dz^m}f|_{z=0}$ can be computed in time $n^{O(1)}$ for each $m=0,\ldots,n$.
\end{lemma}

The quality of the approximation \eqref{eq:taylor} depends on the complex roots of $p$.
\begin{lemma}\label{lem:approx}
Let $p$ be a polynomial of degree $d$ such that $p(0)\neq 0$ for all $z$ such that $|z|\leq M$ for some $M>0$  and let $t\in \C$ such that $|t|<M$.
Then
for $n=O(\ln1/\eps)$ we have 
\begin{equation}
\big||T_n(f)(t)|-\ln|p(t)|\big|\leq d\eps,   \label{eq:add approx}
\end{equation}
and for $n=O(\ln d/\eps)$ we have
\begin{equation}
e^{-\eps}\leq\left|\frac{e^{T_n(f)(t))}}{p(t)}\right|\leq e^\eps\text{ and the angle between $p(t)$ and $e^{T_n(f)(t)}$ is at most $\eps$.}\label{eq:mult approx}
\end{equation}
\end{lemma}
\begin{proof}
Let $q:=|t|/M$. We will first show 
\begin{equation}
|f(t)-T_n(f)(t)|\leq\frac{d q^{n+1}}{(n+1)(1-q)}.	\label{eq:approx}
\end{equation}
The proof of \eqref{eq:approx} is quite similar to the proof of Lemma 1.2 from \cite{BS14a}, but for completeness we will give it here.

Since $p(0)\neq 0$ we may write
\begin{equation}
p(z)=p(0)\prod_{i=1}^d\left(1-\frac{z}{\zeta_i}\right),\label{eq:pol}
\end{equation}
where $\zeta_1,\ldots,\zeta_d\in \C$ are the roots of $p$.
By assumption we know that $|\zeta_i|\geq M$ for each $i=1,\ldots,d$.
This implies that
\begin{equation}
f(z)=\ln p(z)=\ln p(0)+\sum_{i=1}^d \ln \left(1-\frac{z}{\zeta_i}\right) \label{eq:expand f}
\end{equation}
for any $z$ such that $|z|\leq |t|$.
Using the standard Taylor expansion for the principal branch of the logarithm, we obtain
\begin{equation}
\ln \left(1-\frac{t}{\zeta_i}\right)=-\sum_{m=1}^n \frac{1}{m}\frac{t^m}{{\zeta_i}^m} +R_n,\label{eq:taylor log}
\end{equation}
where $R_n$ satisfies
\[
|R_n|=\left|\sum_{m=n+1}^\infty \frac{1}{m}\frac{t^{m}}{{\zeta_i}^m}\right|\leq \frac{1}{(n+1)}\sum_{m=n+1}^\infty \frac{|t|^m}{M^m}\leq \frac{q^{n+1}}{(n+1)(1-q)},
\]
since $q=|t|/M<1$.
Noting that 
\[
-\frac{1}{m}\sum_{i=1}^d \frac{t^m}{{\zeta_i}^m}=\frac{t^m}{m!}\left.\frac{d^m}{dz^m}f(z)\right|_{z=0},
\]
\eqref{eq:approx} now follows by combining \eqref{eq:expand f} with \eqref{eq:taylor log} and using the bound on $|R_n|$. 

Take $n=C(\ln d/\varepsilon)$ (where $C\geq (\ln \new{1/}q)^{-1}$ is large enough so that $1/n\leq 1-q$).
Then the right-hand side of \eqref{eq:approx} is at most $\eps$. 
Write $z=T_n(f)(t)$. Then we have 
\new{$|e^{f(t)-z}|=|e^{f(t)-z}|\leq e^{|f(t)-z|}\leq e^{\eps}$}
and \new{similarly} $|e^{-z+f(t)}|\leq e^{\eps}$. 
(This follows form the fact that for a complex number $y=a+bi$, we have $|e^y|=e^a$ and $|a|\leq |y|$.)
Moreover, \new{the angle between $e^{z}$ and $e^{f(t)}$ is bounded by $|\Im \ln e^{z-f(t)}|\leq |\ln e^{z-f(t)}|\leq \eps$.}
This shows \eqref{eq:mult approx}.

To see \eqref{eq:add approx}, take $n=O(\ln1/\eps)$ such that the right-hand side of \eqref{eq:approx} at most $d\eps$.
Then we have $|e^{z}|\leq |e^{f(t)}|e^{d\eps}$ and $|e^{f(t)}|\leq |e^{\new{z}}|e^{d\eps}$, which is equivalent to \eqref{eq:add approx}.
\end{proof}

\subsection{Approximating tensor network contractions}
In this section we will combine Theorem \ref{thm:absence of roots} with Lemmas \ref{lem:p implies f} and \ref{lem:approx} to give a quasi-polynomial time algorithm for the computation of tensor network contractions for a large class of tensor networks.
In particular, Theorem \ref{thm:algorithms} will be proved here.

Consider a graph $G=(V,E)$ and $h\in S_G$. Suppose that \new{$h$ satisfies 
\begin{equation}\label{eq:condition h}
|h^v(\alpha)-1|\leq\frac{1}{1.01}\frac{\beta^*(\Delta+1)}{2(\Delta(G)+1)}
\end{equation}
} 
for all $v\in V$ and $\alpha\in \N^k_{\deg(v)}$.
Let $I=(I^v)_{v\in V}\in S_G$ be equal to the all-ones tensor for each $v\in V$, i.e., $I^v_{\alpha}=1$ for all $\alpha\in \N^k_{\deg(v)}$.
Define a univariate polynomial $q$ of degree $|V|$  by,
\begin{equation}
q(z):=p(G)(I+z(h-I)).	\label{eq:part pol}
\end{equation}
Note that $q(1)=p(G)(h)$.
So we are interested in computing the value of $q$ at $z=1$.
By Corollary \ref{thm:absence of roots}a we have that for any $|z|\leq 1.01$, $q(z) \neq 0$.
So by Lemma \ref{lem:approx} we can compute a fast $\eps$-approximation to $p(G)(h)$ provided we can compute the $m$th derivative of $q$ at $z=0$ efficiently.
This can be done in time $|V|^{O(m)}$ for each $m$.
Indeed, for any $m$, 
\begin{align}
\frac{d^m}{dz^m}q(z)|_{z=0}&=\sum_{\phi:E\to [k]}\frac{d^m}{dz^m}\prod_{v\in V}(I\new{+}z(h-I))(\phi(\delta(v)))\big|_{z=0} \nonumber
\\
&=\sum_{\phi:E\to [k]}\sum_{U\subseteq V}\prod_{v\in U}(h-I)(\phi(\delta(v)))\frac{d^m}{dz^m}z^{|U|}\big|_{z=0}\nonumber
\\
&=\sum_{U\subseteq V, |U|=m}\sum_{\phi:E\to [k]}m!\prod_{v\in U}(h-I)(\phi(\delta(v))). \label{eq:computing derivative}
\end{align}
Let us denote by $E(U)$ the edges in $E$ that are incident with some vertex of $U$ for $U\subseteq V$.
Then \eqref{eq:computing derivative} is equal to
\[
m!\sum_{U\subseteq V, |U|=m}k^{|E\setminus E(U)|}\sum_{\phi:E(U)\to [k]}\prod_{v\in U}(h-I)(\phi(\delta(v))),
\]
and hence can be computed in time $|V|^m k^{m\Delta(G)}=|V|^{O(m)}$. 
By Lemmas \ref{lem:p implies f} and \ref{lem:approx}, we have the following result, which clearly implies Theorem \ref{thm:algorithms}.
\begin{theorem}\label{thm:approx tensor network}
Let $\eps>0$.
Then for any graph $G=(V,E)$ and any $h\in S_G$ such that \new{$h$ satisfies} \eqref{eq:condition h} for all $v\in V$ and $\alpha\in \N^k_{\deg(v)}$, we can compute a multiplicative $\eps$-approximation to $p(G)(h)$ in time $|V|^{O(\ln |V|-\ln \eps)}$, and an additive $\eps|V|$-approximation to $\ln |p(G)(h)|$ in time $|V|^{O(\ln(1/\eps))}$. 
\end{theorem}
\begin{rem}
Let $\Delta\in \N$ and fix an edge-coloring model $h$ such that \new{$h$ satisfies} \eqref{eq:condition h} for all $d\leq \Delta$.
Then by Corollary \ref{thm:absence of roots}a we have that $|p(G)(h)|\geq e^{c|V|}$ for some constant $c>0$ (assuming $G$ is connected).
So Theorem \ref{thm:approx tensor network} provides a deterministic FPTAS to compute $\ln |p(G)(h)|$ on graphs with maximum degree at most $\Delta$.
\end{rem}

\begin{rem}
Theorem \ref{thm:approx tensor network} of course also applies to tensor networks that strictly satisfy the conditions of Corollaries \ref{thm:absence of roots}a and b, but we leave the details to the reader.
\end{rem}

\subsection{Approximating evaluations of graph polynomials}\label{sec:exp type}
In this section we will combine a result of Csikv\'ari and Frenkel \cite{CF12} with Lemma \ref{lem:p implies f} and Lemma \ref{lem:approx} to give a quasi-polynomial approximation scheme for evaluating a large class of graph polynomials, including the Tutte polynomial.

Let $\chi:\G\to\C$ be a graph parameter.
Define for a graph $G=(V,E)$ the graph polynomial $p_\chi(G)(z):=\sum_{k=1}^{|V|}\chi_k(G)z^k$, where for $k=1,\ldots,|V|$,
\begin{equation}
\chi_k(G):=\sum_{\substack{V=V_1\cup\ldots \cup V_k\\V_i\subseteq V \text{nonempty}\\ V_i\cap V_j=\emptyset \text{ if } i\neq j}}\prod_{i=1}^k\chi(G[V_i])	\label{eq:exp type}
\end{equation}
(here $G[V_i]$ is the subgraph of $G$ induced by $V_i$).
So $\chi(G)$ is a polynomial of degree at most $|V|$ with zero constant term and the coefficient of $z$ is equal to $\chi(G)$. If $\chi(K_1)\neq 0$, then the degree of $p_\chi$ is equal to $|V|$.
Csikv\'ari and Frenkel \cite{CF12} call graph polynomials that arise this way \emph{graph polynomials of exponential type}. 
(In fact, they give a different definition, but show that the definition above is equivalent to it.)

A graph polynomial of exponential type $p(G)(z)=\sum_{i=1}^na_i(G)z^i$ is called of \emph{bounded exponential type} if it is monic of degree $n=|V(G)|$ for each $G$ and if there exists a function $R:\N\to [0,\infty)$ such that for each $t\geq 1$ and $v\in V$ we have
\[
\sum_{\substack{S\subseteq V(G)\\v\in S,|S|=t}}|a_1(G[S])|\leq R(\Delta)^{t-1}
\]
Csikv\'ari and Frenkel proved that these polynomials have bounded roots on $\G_\Delta$:
\begin{theorem}[Csikv\'ari and Frenkel \cite{CF12}]
Let $p$ be a graph polynomial of bounded exponential type. For any $\Delta\in \N$ there exists a constant $c>0$ (one may take $c\approx 7.30319$) such that for each $G\in \G_\Delta$, the roots of $p(G)$ are contained in the set $\{z\in \C: |z|< cR(\Delta)\}$.
\end{theorem}

Csikv\'ari and Frenkel showed that many well known graph polynomials are of bounded exponential type such as the chromatic polynomial, the Tutte polynomial, the modified matching polynomial and the adjoint polynomial.
Let us call a graph parameter $\chi:\G\to \C$ \emph{efficient} if it can be computed in time $2^{O(|V(G)|)}$.
We will show that we can approximate evaluations of graph polynomial of bounded exponential type using Lemmas \ref{lem:p implies f} and \ref{lem:approx}, provided $\chi$ is efficient.

\begin{theorem}\label{thm:exp type approx}
Let $\Delta\in \N$ and let $G=(V,E)$ be a graph with maximum degree at most $\Delta$.
Let $\chi:\G\to \C$ be an efficient graph parameter with $\chi(K_1)=1$ and suppose that all roots of $p_\chi(G)$ have absolute value at most $c$, for some $c>0$.
Fix $\kappa>1$. Then for any $\eps>0$ and $x\in \C$ with $|x|\geq \kappa c$ we can compute a multiplicative $\eps$-approximation to $p_\chi(G)(x)$ in time $|V|^{O(\ln |V|-\ln \eps)}$, and an additive $\eps|V|$-approximation to $\ln|p_\chi(G)(x)|$ in time $|V|^{O(\ln1/ \eps)}$. 
\end{theorem}
\begin{proof}
Define the graph polynomial $q$ by 
\begin{equation}
q(G)(z):=z^{|V|}p_\chi(G)(1/z). \label{eq:transform}
\end{equation}
Note that $q(G)(z)=\sum_{i=0}^{|V|}\chi_k(G)z^{|V|-k}$.
Consequently, the roots of $q$, for $G\in \G_\Delta$, are contained in the set $A:=\{z\in \C\mid |z|> \frac{1}{c}\}$.
Note that $q(G)(0)=\chi_{|V|}=1$, as $\chi(K_1)=1$.
We will now show how to compute the derivatives of $q$.

Let us write $n:=|V|$ and define for $k=0,1,\ldots, n$, $\mathcal{P}_k$ to be the set of partitions of $V$ into exactly $k$ nonempty sets.
Then
\begin{equation}
\left.\frac{d^m}{dz^m}q(z)\right|_{z=0}=\left.\sum_{k=0}^n \sum_{\{V_1,\ldots,V_k\}\in \mathcal{P}_k} \prod_{j=1}^k \chi(G[V_j])\frac{d^m}{dz^m} z^{n-k}\right|_{z=0}.\label{eq:derivative exp type}
\end{equation}
For $m\neq n-k$ the contribution will be zero. So \eqref{eq:derivative exp type} is equal to
\begin{equation}
m!\sum_{\{V_1,\ldots,V_{n-m}\}\in \mathcal{P}_{n-m}} \prod_{j=1}^{n-m} \chi(G[V_j]).\label{eq:contribution m}
\end{equation}
No more than $m$ sets in $\mathcal{P}_{n-m}$ can have size \new{greater than or equal} to $2$.
So to enumerate $\mathcal{P}_{n-m}$, we first select a set of $n-2m$ vertices and then find all partitions into exactly $m$ sets of the remaining $2m$ vertices. This gives a total of $n^{O(m)}$ steps.
Since $\chi$ is efficient, the sum \eqref{eq:contribution m} can be computed in time $n^{O(m)}$.

Let now $t=1/x$. Then $|t|<\frac{1}{c\kappa}$.
Lemmas \ref{lem:p implies f} and \ref{lem:approx} then imply that we can compute numbers $\xi_1$ in time $n^{O(\ln (n/\eps))}$ such that $e^{-\eps}\leq |\frac{\xi_1}{q(t)}|\leq e^{\eps}$ and such that the angle between $\xi_1$ and $q(t)$ is at most $\eps$, and $\xi_2$ in time  $n^{O(1/\eps)}$ such that $\big|\ln |q(t)|-\xi_2\big|\leq \eps n$.
Then $\alpha_1:=x^{n}\xi_1$ is a multiplicative $\eps$-approximation to $p_\chi(x)$ and $\alpha_2:=\xi_2+n\ln|x|$ is an additive $\eps n$-approximation to $\ln|p_\chi(x)|$.
This finishes the proof.
\end{proof}
\new{Csikv\'ari and Frenkel \cite[Remark 5.4]{CF12} proved that the Tutte polynomial is of exponential type (it is clearly multiplicative).
It also follows from their characterization of graph polynomials of exponential type, cf. \cite[Theorem 5.1]{CF12} that the associated graph invariant $\chi$ is efficient.
Jackson, Procacci and Sokal \cite[Theorem 1.2 and 1.3]{JPS13} showed that for every fixed $v_0\in \C$ there exists a constant $C=C(\Delta,v_0)$ such that the roots of $T(G)(q,v_0)$ for any graph $G$ with maximum degree at most $\Delta$ are contained in a disk of radius $C$.
This implies that Theorem \ref{thm:tutte} follows from Theorem \ref{thm:exp type approx}.}

\section{Sparse graph limits and edge-coloring models}\label{sec:sparse limits}
In this section we will prove Theorem \ref{thm:sparse limit}.
To do this we will use the framework of Csikv\'ari and Frenkel \cite{CF12}.

Let $\Delta\in \N$, and consider for any $\eta\geq 0$ the following collection of edge-coloring models
\[
\mathcal{E}(\Delta,\eta):=\{h:\N^k\to \C: |h(\alpha)-1|< (1-\eta)\new{\frac{\beta^*(\Delta+1)}{2(\Delta+1)}}\text{ for all }\alpha\in \N^k\}.
\]
\new{Recall that $\beta^*(\Delta+1)\geq \beta^*(1)\approx 0.71885$}; see Section \ref{sec:absence}. 
Now fix some small constant $\delta>0$ and fix $h\in \mathcal{E}(\Delta,\delta)$. 
Consider the following two graph polynomials:
\begin{equation}
q(G)(z):=p(G)(I-z(I-h)) \quad \text{ and }\quad \hat q(G)(z):=k^{-|E(G)|}z^{|V(G)|}q(G)(1/z).\label{eq:define pol}
\end{equation}
Observe that $q(G)(1)=p(G)(h)=k^{|E(G)|}\hat q(G)(1)$ and that $\hat q$ is monic.

Since for $|z|<\frac{1}{1-\delta}$ we have $I+z(h-I)\in \mathcal{E}(\Delta,0)$, it follows by Corollary \ref{thm:absence of roots}a that $p(G)(I+z(h-I))\neq 0$ for all graphs of maximum degree at most $\Delta$.
This implies that $q(G)(z)\neq 0$ for all $z\in D_\delta:=\{z\in \C:|z|<\frac{1}{1-\delta}\}$ and all $G\in \G_\Delta$, which in turns implies that for all $G\in \G_\Delta$, all roots of $\hat q(G)$ are contained in the the compact set $K_\delta:=\{z\in \C: |z|\leq 1-\delta\}$.

For a graph $G$ we define $\mu_G$ to be the uniform distribution on the roots of $\hat q(G)$.
Note that for all graphs $G\in\G_\Delta$ the measures $\mu_G$ are all supported on $K_\delta$.
We have the following result:
\begin{theorem}\label{thm:cf}
Let $(G_n)$ be a locally convergent graph sequence.
For any open set $\Omega\subset \R^d$ and any continuous function $f:K_\delta\times \Omega\to \R$ that is harmonic on the interior of $K_\delta$ for any fixed $\omega\in \Omega$, and harmonic on $\Omega$ for any fixed $z\in K_\delta$, the sequence 
\[
\int_{K_\delta} f(z,\omega)d\mu_{G_n}(z)
\]
converges locally uniformly in $\omega\in \Omega$ to a harmonic function on $\Omega$.
\end{theorem}
We will prove Theorem \ref{thm:cf} below.
Let us first note that it implies Theorem \ref{thm:sparse limit} \new{by proving the following more concrete result.}
\new{\begin{theorem}\label{thm:sparse limit precise}
Let $\Delta\in \N$ and let $\delta>0$.
For any edge-coloring model $h\in \mathcal{E}(\Delta,\delta)$, the normalised partition function $n(\cdot)(h):\G_\Delta\to \R$ is estimable.
\end{theorem}}
\begin{proof}
Let $(G_n)$ be a locally convergent graph sequence with $G_n\in \G_\Delta$ for each $n$.
Take $\Omega=\C\setminus K_\delta$ and let $f(z,\omega):=\ln|\omega-z|$ on $K_\delta\times \Omega$. 
Then taking $1\in \Omega$, we have by Theorem \ref{thm:cf} that
\begin{equation}\label{eq:convergent at 1}
\int_{K_\delta} f(z,1) d\mu_{G_n}(z)
\end{equation}
converges.
Now note that \eqref{eq:convergent at 1} is equal to
\begin{align}
\int_{K_\delta} \ln|1-z|d\mu_{G_n}(z)&=\frac{\sum_{\zeta: \hat q(G_n)(\zeta)=0 }\ln|1-\zeta|}{|V(G_n)|}=\frac{\ln|\hat q(G_n)(1)|}{|V(G_n)|}	\nonumber
\\
&=\frac{\ln|k^{-|E(G_n)|}p(G_n,h)|}{|V(G_n)|}=n(G_n)(h)-\frac{|E(G_n)|}{|V(G_n)|}\ln k. \label{eq:convergent at 1b}
\end{align}
Now we just need to observe that since $(G_n)$ is locally convergent, we have that $\frac{|E(G_n)|}{|V(G_n)|}=\frac{\ind(K_2,G)}{2|V(G_n|)}$ converges. 
So \eqref{eq:convergent at 1b} implies that the sequence $(n(G_n)(h))$ is convergent, proving the theorem.
\end{proof}
\begin{rem}
Theorem \ref{thm:sparse limit precise} is stated only for edge-coloring models that are close to the all-ones model. Of course it also applies to edge-coloring models that strictly satisfy the conditions of Corollaries \ref{thm:absence of roots}a and b. In particular, in view of Remark \ref{rem:vertex-coloring}, the result gives a qualitative version of a result of Borgs, Chayes, Kahn and Lov\'asz \cite{BCKL12} on `right convergence'; see also \cite{LML15}.
We however leave the details to the reader.
\end{rem}

It now remains to prove Theorem \ref{thm:cf}.

\begin{proof}[Proof of Theorem \ref{thm:cf}]
We will apply Theorem 4.6 from \cite{CF12}.
Write $\hat q(G)(z)=a_0(G)+a_1(G)z^1+\ldots \new{+}a_{n}(G)z^{n}$ with $n:=|V|$.
The only requirement in \cite[Theorem 4.6]{CF12} that is left to check is that the coefficients $a_{n-i}(\cdot)$ can be expressed as a linear combination of the parameters $\ind(H,\cdot)$ for $i=0,\ldots,m$ for each fixed $m\in \N$. \new{(In fact, Csikv\'ari and Frenkel do not work with induced subgraphs, but with subgraphs, but these notions are equivalent by inclusion exclusion.)}

We will show that for any graph $G=(V,E)\in \G_\Delta$ on $n$ vertices the coefficients of $z^0,\ldots,z^{m}$ of $q(G)$ can be expressed as a linear combination, over graphs $H$ of at most $(\Delta +1)m$ vertices, of the parameters $\ind(H,G)$. 
Since $q(G)(z)=k^{-|E(G)|}\sum_{i=0}^na_{n-i} z^i$, by \eqref{eq:define pol}, this is clearly equivalent to the statement about the coefficients of $\hat q$.

By definition, $q(G)(z)$ is equal to
\[
\sum_{\phi:E\to[k]}\prod_{v\in V}(I-(z(I-h))(\phi(\delta(v)))=\sum_{\phi:E\to[k]}\sum_{i=0}^nz^{i}\sum_{\substack{U\subseteq V\\|U|=i}}\prod_{u\in U}(I-h)(\phi(\delta(v)). 
\]
So we only need to show that we can express 
\begin{equation}
\sum_{\phi:E\to[k]}\sum_{\substack{U\subseteq V\\|U|=i}}\prod_{u\in U}(I-h)(\phi(\delta(v)))=
\sum_{\substack{U\subseteq V\\|U|=i}}\sum_{\phi:E\to[k]}\prod_{u\in U}(I-h)(\phi(\delta(v)))	\label{eq:pf G(U)}
\end{equation}  
as a linear combination of the parameters $\ind(H,G)$ for each $i=0,\ldots,m$.

We need the concept of a \emph{fragment}, which is a pair $(H,\psi)$, where $H$ is a graph and where $\psi$ is a map $\psi:V(H)\to \{0,1,\ldots,\Delta\}$. 
We think of $\psi(u)$ as a number of half edges sticking out of $u$.
For $U\subseteq V$ we let $G(U)$ be the fragment $(G[U],\psi)$ where $\psi(u)$ is equal to the number of edges that connect $u$ with $V\setminus U$.
Clearly, for each $U$ of size $i$ the second sum on the right in \eqref{eq:pf G(U)} only depends on the isomorphism class of the fragment $G(U)$. (An \emph{isomorphism} from a fragment $(H,\psi)$ to a fragment $(H,\psi')$ is an isomorphism $\alpha$ of the underlying graphs such that for each $u\in V(H)$, $\psi(u)=\psi'(\alpha(u))$.)
For a fragment $F=(H,\psi)$ let $E(F)$ denote the set of edges of $F$ including half edges. 
Then define for an edge-coloring model $y$, 
\begin{equation}
p(F)(y):=\sum_{\phi:E(F)\to [k]}\prod_{v\in V(H)}y(\phi(\delta(v))).
\end{equation}
Then we can rewrite \eqref{eq:pf G(U)} as
\begin{equation}
\sum_{\substack{F=(H,\psi)\\ |V(H)|=i}} k^{|E(G)|-|E(F)|}p(F)(I-h)\ind^*(F,G),\label{eq:fragment}
\end{equation}
where the sum runs over fragments and where $\ind^*(F,G)$ denotes the number of sets $U$ of size $|V(H)|$ such that $G(U)$ is isomorphic to $F=(H,\psi)$.

So to show that \eqref{eq:fragment} can be written as a linear combination of the parameters $\ind(H,G)$ for certain graphs $H$, it suffices to show that we can do this for $\ind^*(F,G)$ for any fragment $F$.
To see this, note first that we can write
\[
\ind(H,G)=\sum_{\phi:V(H)\to \{0,1,\ldots,\Delta\}} \ind^*((H,\phi),G).
\]
Next, say that $\psi\leq \kappa$ for two maps $\psi,\kappa:V(H)\to \{0,1,\ldots,\Delta\}$ if $\psi(v)\leq \kappa(v)$ for all $v\in V(H)$. This defines a partial order on $[\Delta]^{V(H)}:=\{\psi:V(H)\to \{0,1,\ldots,\Delta\}\}$.
Let $Z$ be the associated \emph{zeta matrix}; that is, $Z$ is the $[\Delta]^{V(H)}\times [\Delta]^{V(H)}$ matrix defined by,
\[
Z_{\psi,\kappa}=\left\{\begin{array}{l}1 \text{ if } \psi \leq \kappa
													\\
													0 \text{ otherwise. }\end{array}\right.
\]
Let $a\in \C^{[\Delta]^{V(H)}}$ be the vector defined by $a_{\psi}=\ind^*((H,\psi),G)$. Then $\ind(H,G)=(Za)_{\psi_0}$, where $\psi_0$ denotes the all-zero map.
For any fixed $\psi\in [\Delta]^{V(H)}$, $(Za)_\psi$ is equal to $\sum_{H'}\ind(H',G)$, where this sum runs over all graphs $H'$ on $|V(H)|+\sum_{v\in V(H)}\psi(v)$ vertices containing $H$ as an induced subgraph such that for each $v\in V(H)$ this vertex has exactly $\psi(v)$ extra neighbours in $H'$. 
It is well-known (by M\"obius inversion) that $Z$ is invertible. 
Hence $a=Z^{-1}(Za)$, which gives for each fragment $(H,\psi)$ an expression of $\ind^*((H,\psi),G)$ as a linear combination of $\ind(H',G)$, where $H'$ contains at most $(\Delta+1)m$ vertices.
This finishes the proof.
\end{proof}

\section{Concluding remarks}\label{sec:remarks}
In this paper we built on ideas of Barvinok and Sober\'on \cite{BS14a} to find regions where the partition functions of an edge-coloring models does not evaluate to zero. We moreover used a technique of Barvinok \cite{B14a} to transform these zero-free regions into quasi-polynomial time approximation algorithms for computing partition functions, as well as for evaluating graph polynomials.
This work leads to some questions:
\begin{itemize}
\item As is indicated by Example \ref{ex:ind}, our results are not optimal. 
Can one determine more precisely the region where the partition functions does not evaluate to zero on bounded degree graphs?
\new{Recent work of Barvinok \cite{B16} indicates that also zero-free regions that are not shaped like a ball are interesting. In \cite{B16} he used a strip argument to find a quasi polynomial time algorithm to approximate the permanent of a real matrix $A$ for which $\delta \leq A_{i,j}\leq 1$ for some fixed $\delta>0$.}
\item Can zero-free regions be used in any way to yield (randomized) polynomial time approximation algorithms?
\end{itemize}


\begin{ack}
I thank Alexander Barvinok for useful suggestions for the proof of Theorem \ref{thm:absence of roots}.
I am also grateful to Bart Litjens, Sven Polak, Lex Schrijver and Bart Sevenster for their comments.
Moreover, I am grateful for useful comments of the anonymous referees.

The research leading to these results has received funding from the European Research Council
under the European Union's Seventh Framework Programme (FP7/2007-2013) / ERC grant agreement
n$\mbox{}^{\circ}$ 339109 as well as from a personal NWO VENI grant.
\end{ack}

\end{document}